\documentclass[twoside,12pt,leqno]{amsart}
\usepackage{amssymb,latexsym,verbatim} % verbatim for ``comment''
\usepackage{hyperref} 
\hypersetup{citecolor=red, linkcolor=blue, colorlinks=true}
\theoremstyle{plain}
\newtheorem{theorem}{Theorem}%[section]
\newtheorem{lemma}[theorem]{Lemma}

\newtheorem{remark}[theorem]{Remark}

\renewcommand{\geq}{\geqslant}
\renewcommand{\leq}{\leqslant}

\begin{document}

\newcommand\A{\mathcal{A}}
\newcommand\BigO{\textup{O}}
\newcommand{\C}{\textup{C}}
\newcommand{\diag}{\textup{diag}}
\newcommand{\eps}{\varepsilon}
\newcommand{\F}{\mathbb{F}}
\newcommand\GL{\textup{GL}}
\newcommand\qbinom[2]{\genfrac{[}{]}{0pt}{}{#1}{#2}}
\newcommand{\h}[1]{{\widehat{#1}}}
\def\hi#1{\parskip=0pt\setlength{\hangindent}{8mm}\noindent{#1}\par}
\newcommand\im{\textup{im\,}}
\newcommand\Irr{\textup{Irr}}
\renewcommand{\l}{\lambda}
\newcommand{\M}{\textup{M}}
\newcommand{\Mon}{\mathcal{M}}
\newcommand{\Magma}{\text{\sc Magma}}
\newcommand{\Mat}{\textup{Mat}}
\newcommand{\MeatAxe}{\text{\sc Meat-axe}}
\newcommand\n{\newline}
\newcommand{\N}{\mathbb{N}}
\newcommand{\NC}{N\kern-1.2ptC}
\newcommand{\norm}[1]{\parallel\kern-1pt #1\kern-1pt\parallel}
\newcommand{\Prob}{\textup{Prob}}
\newcommand\OO{\textup{O}}
\newcommand\oo{\textup{o}}
\newcommand\ord{\textup{ord}}
\newcommand\U{\textup{U}}
\def\st{\textup{\Large$*$}}
\newcommand{\type}{\textup{type}}
\newcommand{\Z}{\mathbb{Z}}

\hyphenation{Frob-enius}

\begin{comment}
Dear Coauthors,

Below are the reports on our paper (v12). The criticisms are not mathematical.
I have made changes (v13) with the exception of ``some clear-cut motivation
for this work''. Could you review my changes please? Also can you help
with Eamonn's request for better motivation?

p1 Deleted last sentence of abstract.

p1 Added citation to Parker's 1984 paper (and listed in references p10).

p1 Attempted to make introduction less ``breathless''.

p4 Rewrote Acknowledgement to minimize use of term ``nth author''. I left
the reference on p1 to ``Neumann and the fourth author'' unchanged since
Cheryl wanted it this way (I agree it is more modest). Dissenting opinions?

p5 Refined definition of $X$ in Notation B (insist that $A$ and $B$ are cyclic)
p5 Deleted $2\times 2$ in paragraph after Notation B

p7 last paragraph of Section 3. Rewrote the ``heuristic argument''.

p10 Alphabetized references: M before N.  Modified affiliation
Question: Should Mathematica be cited as [M], or [W] for Wolfram?

Once we have agreed on changes, I am happy to resubmit.

In the light of recent remarks by Gowers and others:
http://gowers.wordpress.com/2012/02/26/elseviers-open-letter-point-by-point-and-some-further-arguments/
I will reconsider supporting J. Algebra in the future. Is this an issue
with Australian colleagues?

Best wishes to all,
Stephen
\end{comment}

\begin{abstract}
Let $\M(V)=\M(n,\F_q)$ denote the algebra of $n\times n$ matrices over $\F_q$,
and let $\M(V)_U$ denote the (maximal reducible) subalgebra that normalizes
a given $r$-dimensional subspace $U$ of $V=\F_q^n$ where $0<r<n$. We prove
that the density of non-cyclic matrices in $\M(V)_U$ is
at least $q^{-2}\left(1+c_1q^{-1}\right)$, and at most
$q^{-2}\left(1+c_2q^{-1}\right)$, where~$c_1$ and~$c_2$ are constants
independent of $n,r$, and $q$. The constants $c_1=-\frac43$ and
$c_2=\frac{35}3$ suffice.
%These bounds imply that the probability, $P$, that a uniformly random element
%of $\M(V)_U$ is non-cyclic satisfies $0.08<P$, and $P<0.55$ if $q>2$.

%For all $n>r>0$ and all prime powers~$q$
%the lower bound is at least $\frac{1}{3 q^2} > 0.08$, and for all~$q>2$
%the upper bound is at most $\frac{44}{9q^2}<0.55$.
\end{abstract}

\title[Proportion of cyclic matrices in maximal reducible matrix algebras]{Proportion of cyclic matrices in\\ maximal reducible matrix algebras}
\author{Scott Brown, Michael Giudici, S.\,P. Glasby, and Cheryl E. Praeger}

\address[Brown]{
Centre for Mathematics of Symmetry and Computation,
School of Mathematics and Statistics,
University of Western Australia,
35 Stirling Highway,
Crawley 6009,~Australia. {\tt scott.brown@graduate.uwa.edu.au}
Current address: TSG Consulting, Perth 6000, Australia.}
\address[Giudici]
{Centre for Mathematics of Symmetry and Computation\\
School of Mathematics and Statistics\\
University of Western Australia\\
35 Stirling Highway\\
Crawley 6009, Australia. {\tt Michael.Giudici@uwa.edu.au, \href{http://www.maths.uwa.edu.au/~giudici/}{http://www.maths.uwa.edu.au/$\sim$giudici/}}}
\address[Glasby]{
Department of Mathematics,
Central Washington University,
WA 98926, USA. Also affiliated with The Faculty of Information
Sciences and Engineering, University of Canberra, ACT 2601, Australia. {\tt GlasbyS@gmail.com, \href{http://www.cwu.edu/~glasbys/}{http://www.cwu.edu/$\sim$glasbys/}}}.
\address[Praeger]
{Centre for Mathematics of Symmetry and Computation\\
School of Mathematics and Statistics\\
University of Western Australia\\
35 Stirling Highway\\
Crawley 6009, Australia. Also affiliated with King Abdulaziz University, Jeddah, Saudi Arabia. {\tt Cheryl.Praeger@uwa.edu.au, \href{http://www.maths.uwa.edu.au/~praeger}{http://www.maths.uwa.edu.au/$\sim$praeger} }}

\maketitle
\centerline{\noindent AMS Subject Classification (2010): 15B52, 60B20, 68W40}

\section{The main result}\label{S:Main}

The \MeatAxe\ is an algorithm often used to test whether a given group
or algebra of matrices over a finite field acts irreducibly on the
underlying vector space, see \cite{P,HR,NP2}. It uses random selection
to find a `good' matrix, and if successful is able to determine
whether the action is reducible or irreducible. One definition of a
`good' matrix in this context is a {\it cyclic} matrix. (A matrix is
cyclic if its characteristic and minimal polynomials are
equal.) The density of cyclic matrices in absolutely irreducible
groups and algebras is constrained by the following result of Neumann
and the fourth author \cite[Theorem~4.1]{NP}. The probability
$P_{d,q}:=\Prob(\textup{$X\in\M(d,\F_q)$ is non-cyclic})$ satisfies
\begin{equation}\label{E:NP}
  \frac{q^{-3}}{1+q^{-1}}<P_{d,q}<\frac{q^{-3}}{(1-q^{-1})(1-q^{-2})}
  \qquad\textup{for all $d\geq2$ and $q\geq2$}.
\end{equation}
Thus $\frac{2q^{-3}}{3}\leq P_{d,q}\leq\frac{8q^{-3}}{3}$, so
$P_{d,q}=\Omega(q^{-3})$ for $d\geq2$. If $d=1$, then $P_{1,q}=0$ because each
$1\times1$ matrix is cyclic. Bounds on the proportion of non-cyclic matrices in
irreducible-but-not-absolutely-irreducible matrix algebras are also available
in \cite{NP}.

This note shows that cyclic matrices are less dense in maximal
{\it reducible} matrix algebras than full matrix algebras, with density
$1-c(q)q^{-2}$ rather than $1-c'(q)q^{-3}$ where $c(q),c'(q)$ are bounded
functions. We do not know how to estimate the density $\delta$ of cyclic
matrices in arbitrary non-maximal reducible algebras.
Since $0\leq\delta\leq1$, our lower bound
$q^{-2}\left(1+c_1q^{-1}\right)<\delta$ is unhelpful if $c_1<-q$ for some
choice of~$q$. Similarly, our upper bound
$\delta<q^{-2}\left(1+c_2q^{-1}\right)$ is unhelpful if
$c_2>q(q^2-1)$. We go to some effort to find helpful bounds
for all values of $q$. While motivated by a complexity analysis of the
\MeatAxe\ algorithm, we feel that this problem has broader interest.

A modification of Norton's Irreducibility Test, called the Cyclic
Irreducibility Test, was presented in \cite{NP2}.  It was shown to be
a Monte Carlo algorithm that proved irreducibility of a finite
irreducible matrix algebra $\A$ provided a \emph{cyclic pair}
was found, that is a pair $(v,X)$ where $X$ is a cyclic matrix in $\A$,
and $v$ is a cyclic vector for $X$.  It was hoped that cyclic
pairs in reducible matrix algebras, if such exist, could be used to
construct a proper $\A$-invariant subspace. However, it
was not known which reducible algebras $\A$ might contain a
sufficiently high proportion of cyclic matrices to make this approach
worth exploring. In this paper we prove that finite maximal
reducible matrix algebras do indeed have a
plentiful supply of cyclic elements, with the proportion slightly less
than that for the full matrix algebra. A variant of the Cyclic
Irreducibility Test is given in~\cite[p.\;141]{B}.

\goodbreak
\noindent
{\bf Notation A.}\quad 
The following notation will be used throughout the paper.\par
\hi{$F=\F_q$ a  finite field with $q$ elements;}
\hi{$V=F^n$ the $F$-space of $1\times n$ row vectors;}
\hi{$U$ a fixed $r$-dimensional subspace of $V$ where $0<r<n$;}
\hi{$\M(V)=\M(n,F)=F^{n\times n}$ the $F$-algebra of all $n\times n$ matrices
  over $F$;}
\hi{$\GL(V)$ the group of units of $\M(V)$: isomorphic to the general linear
  group $\GL(n,q)$;}
\hi{$\M(V)_U$ the stabilizer in $\M(V)$ of $U$: isomorphic to the algebra of matrices
  $X=\left(\begin{smallmatrix}A&0\\C&B\end{smallmatrix}\right)$ with
  $A\in F^{r\times r}$, $B\in F^{(n-r)\times(n-r)}$, and
  $C\in F^{(n-r)\times r}$;}
\hi{$\GL(V)_U$ the group of units of $\M(V)_U$ comprising all $X$ with
  $\det(X)\kern-1pt=\kern-1pt\det(A)\det(B)\neq0$.}
\setlength{\parskip}{5pt}

\begin{theorem}\label{T:Main}
Suppose that $0<r<n$ and $U$ is an $r$-dimensional subspace of $V:=\F_q^n$.
Then there exist constants $c_1,c_2$, independent of $n,r,q$, such that
the probability that a uniformly distributed random matrix
$X\in\M(V)_U$ is non-cyclic satisfies
\[
  q^{-2}(1+c_1q^{-1})\leq\Prob(\textup{$X\in\M(V)_U$ is non-cyclic})
  \leq q^{-2}(1+c_2q^{-1}).
\]
The constants $c_1=-\frac{4}{3}$ and $c_2=\frac{35}{3}$ suffice.
\end{theorem}

\begin{table}[!ht]
\[
\begin{array}{cc}
\hline
\dim U&\mbox{Proportion of cyclic matrices in $\M(V)_U$ as $\dim(V)\to\infty$}\\ \hline
1&1-q^{-2}-2q^{-3}-\hspace{0.17cm}q^{-4}\hspace{1.3cm}+\hspace{0.18cm}2q^{-6}
  +\hspace{0.18cm}3q^{-7}\hspace{0.08cm} +\textup{lower terms}\\
2&1-q^{-2}-4q^{-3}-\hspace{0.17cm}q^{-4}+4q^{-5}+\hspace{0.18cm}5q^{-6}
  +\hspace{0.18cm}4q^{-7}\hspace{0.05cm}+\textup{lower terms}\\
3&1-q^{-2}-4q^{-3}-3q^{-4}+4q^{-5}+11q^{-6}+\hspace{0.18cm}8q^{-7}
  +\textup{lower terms}\\
4&1-q^{-2}-4q^{-3}-3q^{-4}+2q^{-5}+11q^{-6}+14q^{-7}+\textup{lower terms}\\
5&1-q^{-2}-4q^{-3}-3q^{-4}+2q^{-5}+\hspace{0.18cm}9q^{-6}+14q^{-7}
  +\textup{lower terms}\\
6&1-q^{-2}-4q^{-3}-3q^{-4}+2q^{-5}+\hspace{0.18cm}9q^{-6}+12q^{-7}
  +\textup{lower terms}\\
7&1-q^{-2}-4q^{-3}-3q^{-4}+2q^{-5}+\hspace{0.18cm}9q^{-6}+12q^{-7}
  +\textup{lower terms}\\ \hline
\end{array}
\]
\caption{Proportions of cyclic matrices in $\M(V)_U$ as $\dim(V)\to\infty$.}\label{tm}
\end{table}

\begin{table}[!ht]
\[
\begin{array}{cc}
\hline
\dim U&\mbox{Proportion of cyclic matrices in $\GL(V)_U$ as $\dim(V)\to\infty$}\\ \hline
1&1 - q^{-2} - 2q^{-3} \hspace{1.1cm} + \hspace{0.18cm} q^{-5}
+ 3q^{-6 }\hspace{0.05cm} + q^{-7}\hspace{0.2cm} +\textup{lower terms}\\
2&1 - q^{-2} - 3q^{-3} + q^{-4} + 3q^{-5} + 4q^{-6}-2q^{-7} +\textup{lower terms}\\
3&1 - q^{-2} - 3q^{-3} + q^{-4} + 4q^{-5} + 4q^{-6}-5q^{-7} +\textup{lower terms}\\
4&1 - q^{-2} - 3q^{-3} + q^{-4} + 4q^{-5} + 4q^{-6}-6q^{-7} +\textup{lower terms}\\
5&1 - q^{-2} - 3q^{-3} + q^{-4} + 4q^{-5} + 4q^{-6}-6q^{-7} +\textup{lower terms}\\
6&1 - q^{-2} - 3q^{-3} + q^{-4} + 4q^{-5} + 4q^{-6}-6q^{-7} +\textup{lower terms}\\
7&1 - q^{-2} - 3q^{-3} + q^{-4} + 4q^{-5} + 4q^{-6}-6q^{-7} +\textup{lower terms}\\
\hline
\end{array}
\]
\caption{Proportions of cyclic matrices in $\GL(V)_U$ as $\dim(V)\to\infty$.}\label{tgl}
\end{table}

\begin{remark}\label{rem1} {\rm

(a) The lower bound in Theorem~\ref{T:Main} is positive for all
    $q\geq2$, and the upper bound is less than~1 for all $q>2$. With
    more care we may increase~$c_1$ and decrease~$c_2$. However, a new
    argument is needed to give an upper bound less than~1 when $q=2$
    because the first term in (\ref{E:lower}) below is
    $\frac{q^{-2}}{(1-q^{-1})^2}=1$ when~$q=2$.

(b) The bounds in Theorem~\ref{T:Main} in the cases $r=1$ and $r=n-1$ can be
    deduced from results in
    Jason Fulman's paper \cite{F} since in these cases $\GL(V)_U$ is an affine
    group. The first asymptotic estimate for the probability in
    Theorem~\ref{T:Main}, for general values of $r$, was given as the
    main result in the PhD thesis of the first author~\cite{B} where a
    probabilistic generating function was found for the proportion of
    cyclic matrices in $\GL(V)_U$ for a subspace $U$ of fixed
    dimension $r$. The limiting proportions of cyclic matrices in both
    $\GL(V)_U$ and $\M(V)_U$, as $\dim(V)\rightarrow\infty$, were
    proved to be power series in $q^{-1}$ of the form $1-q^{-2}+
    \sum_{i\geq3} \gamma_iq^{-i}$. (In Tables~\ref{tm} and~\ref{tgl} the `lower
    terms' residual was not bounded by a function of $r$ and $q$ in~\cite{B}.
    By contrast, bounding constants independent of $r$, $n$, $q$ are explicit
    in the statement, and proof, of Theorem~\ref{T:Main}.) Exact values for
    these limiting proportions can be determined from the generating
    function for small values of~$r$, and some sample results are given
    in Tables~\ref{tm} and~\ref{tgl}.  These results show that the
    $\gamma_i$ depend mildly on the dimension~$r$ when $i\geq3$. The
    expressions for $r=1,2$ were deduced analytically, and those for
    $3\leq r\leq 7$ were obtained using \textsc{Mathematica}~\cite{Mca}.

(c) Truncating the power series in Table~\ref{tm} suggests (heuristically) that
    the probability in Theorem~\ref{T:Main} `ought' to have the form
    $q^{-2}(1+4q^{-1})$. This is consistent with the constants
    given in Theorem~\ref{T:Main} as
    $-\frac{4}{3}=c_1\leq 4\leq c_2=\frac{35}{3}$.

(d) The PhD thesis of the first author contains analogous results for
    the limiting proportions (as $\dim(V)\to\infty$) of cyclic
    matrices in maximal completely reducible matrix
    algebras~\cite[Theorems 5.2.8 and 5.3.5]{B}, see also the
    unpublished paper~\cite{BGP}. The limiting proportions of separable matrices
    in maximal reducible matrix algebras are described in~\cite[Theorem 6.4.6]{B}.

%(d) The PhD thesis of the first author contains an analogous
%    result~\cite[Theorem~6.4.6]{B} for the limiting proportions (as
%    $\dim(V)\to\infty$) of separable matrices in maximal completely
%    reducible matrix algebras; see also the unpublished
%    paper~\cite{BGP}.  
}
\end{remark}

\begin{proof}[Proof Strategy for Theorem~\textup{\ref{T:Main}}.]
Since $\GL(V)$ acts transitively on the set of $r$-dimensional subspaces
of $V$, the stabilizers of $r$-dimensional subspaces, being conjugate,
all have the same cardinality. Thus it suffices to consider the stabilizer
$\M(V)_U$ of the $r$-dimensional subspace
$U:=\langle e_1,\dots,e_r\rangle$ where
$e_i$ denotes the $i$th row of the $n\times n$ identity matrix $I_n$.
Suppose that
$X=\left(\begin{smallmatrix}A&0\\C&B\end{smallmatrix}\right)\in\M(V)_U$
is non-cyclic. Exactly one of the following holds:
\begin{itemize}
  \item[(i)] $A$ is non-cyclic, or
  \item[(ii)] $A$ is cyclic, and $B$ is
    non-cyclic, or
  \item[(iii)] $A\in\M(U)$ and $B\in\M(V/U)$ are cyclic, and $X\in\M(V)_U$
    is non-cyclic.
\end{itemize}
Denote by $n_1,n_2,n_3$ the number of {\it non-cyclic} $X\in\M(V)_U$
satisfying the pairwise mutually exclusive cases (i), (ii), and (iii),
respectively. The desired probability is $\pi=\pi_1+\pi_2+\pi_3$ where
$\pi_i:=\frac{n_i}{|\M(V)_U|}$, and $|\M(V)_U|=q^{n^2-nr+r^2}$.

The cases when $r=1$ or $n-1$ can be handled separately. Suppose $1<r<n-1$.
The probability $\pi_1$ that $A$ is non-cyclic is $\Omega(q^{-3})$
by (\ref{E:NP}). Since the events `$A$ is cyclic' and `$B$~is non-cyclic' are
independent, the probability $\pi_2$ is
$(1-\pi_1)\textup{Prob($B$ non-cyclic)}=\Omega(q^{-3})$.
Explicit upper and lower bounds may be determined by applying (\ref{E:NP}).
The proof is complete once we prove that
$\pi_3=q^{-2}\left(1+\Omega(q^{-1})\right)$.
This is achieved by constructing an upper bound for $n_3$ in
Section~\ref{S:ub}, and a lower bound for $n_3$ in Section~\ref{S:lb}.
\end{proof}

Bounds for the density of non-cyclic matrices in the group $\GL(V)_U$ can be
deduced from those for the density in the algebra $\M(V)_U$.
Dividing by $|\GL(V)_U|$ instead of $|\M(V)_U|$ is not problematic
since $|\GL(V)_U|=|\M(V)_U|\left(1+\Omega(q^{-1})\right)$.
An upper bound for non-cyclic matrices in $\M(V)_U$
is also an upper bound for non-cyclic matrices in $\GL(V)_U$  as
$\GL(V)_U\subseteq\M(V)_U$. A lower bound for non-cyclic matrices in $\GL(V)_U$
needs to be altered to ensure that only
{\it invertible} non-cyclic matrices are counted. This requires only
minor modifications to Section~\ref{S:lb}.
Since the \MeatAxe\ is more commonly concerned with algebras and not groups,
we leave this modification to an interested reader.

\medskip\noindent \textbf{Acknowledgements:}\quad The authors are
grateful to Peter Neumann for his advice during many helpful
discussions on this work. The paper grew out of the PhD thesis of
the first author (Brown) undertaken at the University of Western Australia under
the supervision of Giudici and Praeger, and supported by a
University Postgraduate Award. Research for the paper was partially
supported by an Australian Research Council Grant: Giudici and
Praeger are supported by an ARC Australian Research Fellowship
and a Federation Fellowship, respectively.

\section{The upper bound}\label{S:ub}

Let $U$ be an $r$-dimensional subspace of the vector space $V=F^n$ where
$0<r<n$ and $F=\F_q$ is the field with $q$ elements. Let
$\M(V)_U=\{X\in\M(V)\mid UX\subseteq U\}$ be the algebra
of $q^{r^2-rn+n^2}$ matrices that normalize $U$. The goal of this section is to
compute an upper bound for the number, $n_3$, of matrices
$X=\left(\begin{smallmatrix}A&0\\C&B\end{smallmatrix}\right)\in\M(V)_U$
for which $U$ is a cyclic $F[A]$-module, $V/U$ is a cyclic $F[B]$-module,
and $V$ is a {\it non-cyclic} $F[X]$-module.

\noindent
\textbf{Notation B.}\quad
As well as Notation A, the following notation will be used in the paper.\par
\hi{$F[t]$ denotes the ring of polynomials with coefficients in $F$;}
\hi{$X\in\M(V)$ denotes a matrix, and $v\in V$ denotes a (row) vector;}
\hi{$F[X]$ is the subalgebra of $\M(V)$ comprising all polynomials in $X$
  with coefficients in $F$;}
\hi{$vF[X]=\langle v,vX,vX^2,\dots\rangle$ is the {\em cyclic $F[X]$-submodule}
  of $V$ generated by $v$;}
\hi{$f\in\Irr(d,F)$ denotes a monic polynomial of degree $d$ which is
  irreducible in $F[t]$;}
\hi{$c_X$ denotes the characteristic polynomial $c_X(t)=\det(tI_n-X)$ of
  $X\in\M(n,F)$;}
\hi{$m_X$ denotes the minimal polynomial of $X\in\M(n,F)$;}
\hi{$V(f)=\{v\in V\mid vf(X)=0\}=\ker f(X)$;}
\hi{$X=\left(\begin{smallmatrix}A&0\\C&B\end{smallmatrix}\right)\in\M(V)_U$
  denotes a block matrix with $A\in\M(r,F)$ and $B\in\M(n-r,F)$ cyclic,
  $C\in F^{(n-r)\times r}$, and $X$ non-cyclic;}
\hi{$\omega(n,q)=\prod_{i=1}^n(1-q^{-i})$; note that $|\GL(n,q)|=q^{n^2}\omega(n,q)$;}
\hi{$C(a)$ the (row) companion matrix of a polynomial
  $a(t)=t^r+\sum_{i=0}^{r-1}a_it^i$, see (\ref{E:matform}).}
\setlength{\parskip}{5pt}

There exists a monic irreducible polynomial $f\in\Irr(d,q)$ with
$1\leq d\leq\min(r,n-r)$ for which $V_0:=\ker f(X)$ is a non-cyclic
$F[X]$-module. The restriction, $X_0$, of $X$ to~$V_0$ has
minimal polynomial~$f$. Since $U_0:=V_0\cap U$ and
$(V_0+U)/U\cong V_0/U_0$ are cyclic $F[X]$-modules, it follows that $X_0$
is conjugate to the block diagonal matrix $\diag(C(f),C(f))$.
The number of $d$-dimensional subspaces $U_0$ of the $r$-dimensional space
$U$ is given by the $q$-binomial coefficient
\begin{equation}\label{E:bin}
  \qbinom{r}{d}:=\prod_{i=0}^{d-1}\frac{q^r-q^i}{q^d-q^i}
    =q^{d(r-d)}\prod_{i=0}^{d-1}\frac{1-q^{-(r-i)}}{1-q^{-(d-i)}}
    =\frac{q^{d(r-d)}\omega(r,q)}{\omega(d,q)\omega(r-d,q)}.
\end{equation}
It is well known that $\qbinom{r}{d}\in\Z[q]$ is a polynomial
in $q$ over~$\N$, and $\deg(\qbinom{r}{d})=d(r-d)$.

First, choose $d$ in the range $1\leq d\leq\min(r,n-r)$, next choose a monic
$f\in\Irr(d,q)$, then a $2d$-dimensional subspace $V_0$ for which
$\dim(V_0\cap U)=\dim((V_0+U)/U)=d$, then choose a linear transformation
$X_0$ on $V_0$ with minimal polynomial $m_{X_0}=f$
satisfying $U_0X_0\subseteq U_0$, and finally choose
an extension $X$ of $X_0$ to $V$. The number of 4-tuples
$(f,V_0,X_0,X)$ overcounts the number~$n_3$ since different $f$
may give the same $X$. Moreover we shall overcount the number of 4-tuples.

In this paragraph the value of $d$ satisfying $1\leq d\leq\min(r,n-r)$ is fixed.
There are at most $\frac{q^d-q}{d}$ choices for $f$ if $d\geq2$, and $q$
choices if $d=1$. How many choices
are there for $V_0$? First, choose $U_0$ in $\qbinom{r}{d}$ ways, then choose
$U+V_0$, or equivalently choose the $d$-dimensional subspace $(U+V_0)/U$
of the $(n-r)$-dimensional space $V/U$ in $\qbinom{n-r}{d}$ ways. Finally,
choose a $d$-dimensional complement $V_0/U_0$ to the $(r-d)$-dimensional
subspace $U/U_0$ in $(U+V_0)/U$ in $q^{d(r-d)}$ ways. Multiplying shows
that there are exactly $\qbinom{r}{d}\qbinom{n-r}{d}q^{d(r-d)}$
choices for $V_0$. Since $U_0X_0\subseteq U_0$ and $X_0$ has minimal
polynomial~$m_{X_0}=f$, it is conjugate in
$\GL(V_0)_{U_0}$ to the $2\times 2$ block diagonal matrix $\diag(C(f),C(f))$.
The centralizer in $\GL(V_0)_{U_0}$ of $X_0$ has order
$(q^d-1)^2q^d=q^{3d}(1-q^{-d})^2$, and the conjugacy class $X_0^{\GL(V_0)_{U_0}}$
has cardinality
\[
  |X_0^{\GL(V_0)_{U_0}}|=\frac{|\GL(V_0)_{U_0}|}{|C_{\GL(V_0)_{U_0}}(X_0)|}=
  \frac{q^{3d^2}\omega(d,q)^2}{q^{3d}(1-q^{-d})^2}
  =\frac{q^{3(d^2-d)}\omega(d,q)^2}{(1-q^{-d})^2}.
\]
Specifying $X_0$, can be viewed (after a change of basis) as specifying $d$ of
the top $r$ rows, and $d$ of the bottom $n-r$ rows
of $X=\left(\begin{smallmatrix}A&0\\C&B\end{smallmatrix}\right)$.
The remaining rows can be completed in at most $|\M(V)_U|q^{-(r+n)d}$ ways.
This shows
\[
  n_3\leq\sum_{d=1}^{\min(r,n-r)} |\Irr(d,q)|\cdot\qbinom{r}{d}\qbinom{n-r}{d}
      \frac{q^{d(r-d)}}{1}\cdot
      \frac{q^{3(d^2-d)}\omega(d,q)^2}{(1-q^{-d})^2}\cdot
      \frac{|\M(V)_U|q^{-(r+n)d}}{1}.
\]
Equation (\ref{E:bin}) yields $\qbinom{r}{d}\leq\frac{q^{d(r-d)}}{\omega(d,q)}$
and $\qbinom{n-r}{d}\leq\frac{q^{d(n-r-d)}}{\omega(d,q)}$. This, in turn, shows
\[
  n_3\leq\sum_{d=1}^{\min(r,n-r)} |\Irr(d,q)|\cdot
      \frac{q^{d(r-d)+d(n-r-d)+d(r-d)}}{\omega(d,q)^2}\cdot
      \frac{q^{3(d^2-d)}\omega(d,q)^2}{(1-q^{-d})^2}\cdot
      \frac{|\M(V)_U|q^{-(r+n)d}}{1}.
\]
Collecting powers of $q$ gives $q^{-3d}$. Cancelling $\omega(d,q)^2$,
dividing by $|\M(V)_U|$,
and using the inequality $|\Irr(d,q)|\leq\frac{q^d-q}{d}$ for $d\geq2$ gives
\begin{equation}\label{E:lower}
  \begin{aligned}
  \frac{n_3}{|\M(V)_U|}&\leq\frac{q^{-2}}{(1-q^{-1})^2}+
      \sum_{d=2}^{\min(r,n-r)} \frac{q^d-q}{d}\cdot
      \frac{q^{-3d}}{(1-q^{-d})^2}\\
      &<\frac{q^{-2}}{(1-q^{-1})^2}+
      \sum_{d=2}^{\infty} \frac{q^{-2d}}{d(1-q^{-d})}.
  \end{aligned}
\end{equation}
The first term is $\frac{q^{-2}}{(1-q^{-1})^2}\leq q^{-2}(1+6q^{-1})$, and the
infinite sum is less than $\frac{8q^{-4}}{9}$ because
$\frac{1}{d(1-q^{-d})}\leq\frac{1}{2(1-2^{-2})}\leq\frac23$ for $d\geq2$,
and $\sum_{d=2}^\infty q^{-2d}=\frac{q^{-4}}{1-q^{-2}}\leq\frac{4q^{-4}}{3}$. Hence
\[
  \frac{n_3}{|\M(V)_U|}<q^{-2}\left(1+6q^{-1}\right)
    +\frac{4q^{-3}}{9}\leq q^{-2}\left(1+\frac{58q^{-1}}{9}\right).
\]

\section{Counting polynomials}\label{S:Polys}

The goal of this section is to prove a simple combinatorial
result for polynomials over~$\F_q$. This result will be used in 
Section~\ref{S:lb} to prove a lower bound for $n_3$.
Morrison~\cite{M} proves that the density of coprime pairs of polynomials
of degree {\em at most}~$r$ over $\F_q$ is $1-q^{-1}+q^{-2r-1}-q^{-2r-2}$.
A simpler answer exists if the degrees are {\em precisely}~$r$.

\begin{lemma}\label{L:coprime}
Let $\Mon_r$ denote the set of $q^r$ monic polynomials in $\F_q[t]$ of
degree~$r$.
\begin{itemize}
  \item[(a)] The number of coprime ordered pairs $(a,b)$ in $\Mon_r\times\Mon_s$
    is $q^{r+s}(1-q^{-1})$ when $rs>0$, and $q^{r+s}$ when $rs=0$.
  \item[(b)] Fix $f\in\Irr(d,q)$ and suppose $1\leq d\leq\min(r,s)$. Then the
    number of coprime pairs $(a,b)$ in $\Mon_r\times\Mon_s$ satisfying
    $\gcd(f,ab)=1$ is at least $q^{r+s}(1-q^{-1}-2q^{-d}+2q^{-2d})$.
\end{itemize}
\end{lemma}

\begin{proof}
(a) Let $c(r,s)$ denote the number of coprime ordered pairs
$(a,b)\in\Mon_r\times\Mon_s$. The cardinality of $\Mon_r\times\Mon_s$,
{\em viz.} $|\Mon_r|\,|\Mon_s|=q^{r+s}$, can be determined
in a different way.

An ordered pair $(a,b)\in\Mon_r\times\Mon_s$ has $\gcd(a,b)=d$
if and only if $\gcd(\frac ad,\frac bd)=1$. If $\deg(d)=k$, then there
are $q^k$ choices for $d$, and $c(r-k,s-k)$ pairs $(\frac ad,\frac bd)$. Thus
\[
  |\Mon_r\times\Mon_s|=\sum_{k=0}^{\min(r,s)} |\Mon_k|\,c(r-k,s-k),
  \quad\textup{or}\quad
  q^{r+s}=\sum_{k=0}^{\min(r,s)} q^k c(r-k,s-k).
\]
Rearranging gives a recurrence relation
$c(r,s)=q^{r+s}-\sum_{k=1}^{\min(r,s)} q^k c(r-k,s-k)$
with initial conditions $c(r,0)=q^r$, $c(0,s)=q^s$. Induction may be used to
prove $c(r,s)=q^{r+s}(1-q^{-1})$ holds when $rs>0$. (The sum in the recurrence
telescopes to $q^{r+s-1}$.) It is noteworthy that the
probability $c(r,s)/q^{r+s}=1-q^{-1}$ is independent of both $r$ and $s$.

(b) Assume $f\in\Irr(d,q)$ and $1\leq d\leq\min(r,s)$.
We shall underestimate the number of coprime ordered pairs
$(a,b)\in\Mon_r\times\Mon_s$ for which $\gcd(ab,f)=1$. By part~(a) there
are $q^{r+s}(1-q^{-1})$ coprime pairs $(a,b)\in\Mon_r\times\Mon_s$.
The number of $a\in\Mon_r$ divisible by $f$ is $q^{r-d}$, and the number of
$(a,b)\in\Mon_r\times\Mon_s$ with $f\mid a$ and $f\mid b$ is $q^{r+s-2d}$.
Hence $q^{r+s-d}-q^{r+s-2d}$ ordered pairs $(a,b)$ have $f\mid a$ and
$f\nmid b$. The same count holds for ordered pairs $(a,b)$ with
$f\nmid a$ and $f\mid b$. However, some of these ordered pairs may
not be coprime, and therefore
$q^{r+s}(1-q^{-1})-2(q^{r+s-d}-q^{r+s-2d})$ underestimates the number of
coprime $(a,b)$ with $\gcd(f,ab)=1$. Rearranging proves the result.
\end{proof}

A heuristic argument suggests that the matrices
$X=\left(\begin{smallmatrix}A&0\\C&B\end{smallmatrix}\right)\in\M(V)_U$
for which $c_A$ and $c_B$ are not coprime, has density roughly
$q^{-1}$.  An extra factor of $q^{-1}$ arises when we insist that
$X$ is non-cyclic. This is basically because there are $q$
non-cyclic matrices in $\M(V)_U$ when $\dim(V)=2$ and $\dim(U)=1$,
as $C$ must be 0. A rigorous argument is given below.

\section{The lower bound}\label{S:lb}

Fix $X=\left(\begin{smallmatrix}A&0\\C&B\end{smallmatrix}\right)\in\M(V)_U$.
Then $V$ becomes an $F[t]$-module with $v*f(t)=vf(X)$ where the juxtaposition
$vf(X)$ denotes vector-times-matrix multiplication.
We also say that~$V$ is an $F[X]$-module, where $F[X]$ is the subalgebra
of $\M(V)_U$ comprising all polynomials in $X$ over $F$.
In this section we give a lower bound for $n_3$ by underestimating
the number of matrices $X\in\M(V)_U$
which have a {\em unique} non-cyclic primary submodule. For these matrices,
$U$ is a cyclic $F[A]$-module, $V/U$ is a cyclic $F[B]$-module, and
$V$ is a {\it non-cyclic} $F[X]$-module. That is, we are counting certain
$X=\left(\begin{smallmatrix}A&0\\C&B\end{smallmatrix}\right)\in\M(V)_U$
for which $c_A=m_A$, $c_B=m_B$, and $c_X=c_Ac_B\neq m_X$.

Since $U$ and $V/U$ are cyclic, there exist vectors
$u\in U$ and $v+U\in V/U$, generating the respective $F[t]$-modules.
Consider the basis
\begin{equation}\label{E:X}
  u,uX,\dots,uX^{r-1},v,vX,\dots,vX^{n-r-1}
\end{equation}
for $V$. Then $X$ is conjugate in $\GL(V)_U$ to a matrix of the form
$\left(\begin{smallmatrix}A'&0\\C'&B'\end{smallmatrix}\right)\in\M(V)_U$
where
\begin{equation}\label{E:matform}
  A'=\left(\begin{smallmatrix}
    0&1& &0\\&&\ddots&\\0&0& &1\\-a_0&-a_1&\cdots&-a_{r-1}
    \end{smallmatrix}\right),\quad
  B'=\left(\begin{smallmatrix}
    0&1& &0\\&&\ddots&\\0&0& &1\\-b_0&-b_1&\cdots&-b_{n-r-1}
    \end{smallmatrix}\right),\quad
  C'=\left(\begin{smallmatrix}0&0&\cdots&0\\\vdots&\vdots&&\vdots\\0&0&\cdots&0\\
    c_0&c_1&\cdots&c_{r-1}\end{smallmatrix}\right).
\end{equation}
Set $a:=t^r+\sum_{i=0}^{r-1}a_it^i=m_A$,
$b:=t^{n-r}+\sum_{i=0}^{n-r-1}b_it^i=m_B$, and
$c:=\sum_{i=0}^{r-1}c_it^i$.
Then $ua(X)=0$ and $vb(X)=uc(X)$ where $\deg(c)<\deg(a)$.
The matrices $A'$ and $B'$ are called {\it companion matrices} of~$a$
and~$b$ and are abbreviated $C(a)$ and $C(b)$, respectively.

We shall count non-cyclic matrices $X$ for which $a=fg$, $b=fh$,
$f\in\Irr(d,q)$, and $\gcd(f,gh)=\gcd(g,h)=1$.
Note that $V=V(f)\oplus V(gh)$ where $X$ is non-cyclic on $V(f):=\ker f(X)$,
and cyclic on $V(gh)=V(g)\oplus V(h)$. Such matrices $X$ are conjugate in
$\GL(V)_U$ to the block diagonal matrix $X_{f,g,h}:=\diag(C(g),C(f),C(f),C(h))$
for a uniquely determined triple $(f,g,h)$. This fact is needed to establish
a lower
bound for~$n_3$. (Different choices for~$f$ give
different $X_{f,g,h}$ due to our assumption that $V(f)$ is the {\em unique}
non-cyclic primary $F[X]$-submodule of $V$.)
As $X$ is conjugate
in $\GL(V)$ to $\diag(C(f)\oplus C(f),C(gh))$, it follows that
$|C_{\GL(V)_U}(X)|\leq q^{3d}(1-q^{-d})^2(q^{n-2d}-1)$ because
\[
  C_{\GL(V)_U}(X)\leq C_{\GL(V)}(X)
  \cong C_{\GL(V(f))}(C(f)\oplus C(f))\times C_{\GL(V(gh))}(C(gh)).
\]

First, choose $d$ in the range $1\leq d\leq\min(r,n-r)$, next choose a monic
$f\in\Irr(d,q)$, then choose an ordered pair $(g,h)$ satisfying
$\gcd(f,gh)=\gcd(g,h)=1$. By Lemma~\ref{L:coprime}(b), there are at least
\[
  q^{(r-d)+(n-r-d)}(1-q^{-1}-2q^{-d}+2q^{-2d})=q^{n-2d}(1-q^{-1}-2q^{-d}+2q^{-2d})
\]
ordered pairs $(g,h)$. Summing over the relevant triples $(f,g,h)$ gives
\[
  n_3\geq\sum_{d=1}^{\min(r,n-r)}\kern-8pt\sum_{f\in\Irr(d,q)}\;\sum_{(g,h)}
     \frac{|\GL(V)_U|}{|C_{\GL(V)_U}(X_{f,g,h})|}.
\]
But $|\GL(V)_U|=|\M(V)_U|\omega(r,q)\omega(n-r,q)$ and
$|C_{\GL(V)_U}(X)|\leq q^{3d}(1-q^{-d})^2(q^{n-2d}-1)$~so
\[
  \frac{n_3}{|\M(V)_U|}\geq\sum_{d=1}^{\min(r,n-r)}
    |\Irr(d,q)|\frac{\omega(r,q)\omega(n-r,q)}{q^{3d}(1-q^{-d})^2(q^{n-2d}-1)}
    \cdot q^{n-2d}(1-q^{-1}-2q^{-d}+2q^{-2d}).
\]
Euler's pentagonal number theorem shows that
$\omega(\infty,q)>1-q^{-1}-q^{-2}+q^{-5}$.
Therefore
\begin{equation}\label{E:LB}
  \frac{n_3}{|\M(V)_U|}\geq\sum_{d=1}^{\min(r,n-r)} |\Irr(d,q)|
    \frac{q^{-3d}(1-q^{-1}-q^{-2}+q^{-5})^2}{(1-q^{-d})^2(1-q^{-(n-2d)})}
    \cdot (1-q^{-1}-2q^{-d}+2q^{-2d}).
\end{equation}

The number $n_3$ depends on $r$. To emphasize this dependence we
write $n_3(r)$. The automorphism of $\M(V)$ obtained by conjugating by
$e_i\leftrightarrow e_{n-i}$ and then transposing, swaps the maximal reducible
algebras $\M(V)_{U(r)}$ and $\M(V)_{U(n-r)}$. Hence $n_3(r)=n_3(n-r)$.
By swapping $r$ and $n-r$, if necessary, we shall assume that
$\min(r,n-r)=r$. It is convenient to give a sharper lower bound
than (\ref{E:LB}) in the case that $r=1$. The calculation above
has $a=t-\lambda=f$, $g=1$, and $b=fh$ where $h(\lambda)\neq0$.
There are precisely $q^{n-2}(1-q^{-1})$ choices for $h$. (This is
a sharper estimate than given above.) Hence when $r=1$, we have $d=1$.
Since $|\Irr(1,q)|=q$ and
$1-q^{-1}-q^{-2}+q^{-5}=(1-q^{-1})(1-q^{-2}-q^{-3}-q^{-4})$,
a sharper bound than (\ref{E:LB}) for $n\geq3$ is
\begin{equation}\label{E:pi3lb}
  \begin{aligned}
  \frac{n_3(1)}{|\M(V)_U|}&\geq
    \frac{q\cdot q^{-3}(1-q^{-1}-q^{-2}+q^{-5})^2q^{n-2}(1-q^{-1})}
    {(1-q^{-1})^2(q^{n-2}-1)}\\
    &\geq q^{-2}(1-q^{-2}-q^{-3}-q^{-4})^2(1-q^{-1}).
  \end{aligned}
\end{equation}
This bound also holds when $n=2$ and $r=1$, as a direct calculation shows
that $\frac{n_3(1)}{|\M(V)_U|}=q^{-2}$ in this case.

Henceforth assume that $r\geq2$, and hence that $n\geq4$.
The summand in (\ref{E:LB}) with $d=1$ is greater than
\begin{equation}\label{E:d=1}
  q^{-2}(1-q^{-2}-q^{-3}-q^{-4})^2(1-3q^{-1}+2q^{-2})\geq q^{-2}(1-3q^{-1}+4q^{-3}).
\end{equation}
It follows from (\ref{E:LB}) and (\ref{E:d=1}) that
\begin{equation}\label{E:lb}
  \frac{n_3(r)}{|\M(V)_U|}\geq q^{-2}(1-3q^{-1}+4q^{-3})
   \geq q^{-2}\left(1-2q^{-1}\right)
\end{equation}
holds for $r\geq2$. However, the bound (\ref{E:lb}) when $r\geq2$ is always
smaller than the bound~(\ref{E:pi3lb}) when $r=1$. Thus (\ref{E:lb}) gives
a uniform lower bound for all $r$ satisfying $0<r<n$.

\begin{proof}[Proof of Theorem~\textup{\ref{T:Main}}]
Recall the notation $n_i$ and $\pi_i=\frac{n_i}{|\M(V)_U|}$ used in
the `Proof Strategy' in Section~\ref{S:Main}. We shall prove
$q^{-2}(1+c_1q^{-1})\leq\pi\leq q^{-2}(1+c_2q^{-1})$, where
$\pi=\pi_1+\pi_2+\pi_3$ equals $\textup{Prob($X\in\M(V)_U$ is non-cyclic)}$,
and $c_1=-\frac{4}{3}$ and $c_2=\frac{35}{3}$.
As mentioned previously,
we shall assume that $\min(r,n-r)=r$. If $n=2$, then only the $q$ scalar
matrices of the $q^3$ elements of $\M(V)_U$ are non-cyclic. Thus we have
$\pi=\pi_3=q^{-2}$ and the stated bounds 
$q^{-2}(1-\frac{4q^{-1}}{3})\leq q^{-2}\leq q^{-2}(1+\frac{35q^{-1}}{3})$
clearly hold. Suppose now that $n\geq3$. Consider
the case when $r=1$. Then the probability $\pi_1$ that $A$ is non-cyclic is $0$,
and $\frac{2q^{-3}}{3}\leq\pi_2\leq\frac{8q^{-3}}{3}$ by (\ref{E:NP}) because
$n-r\geq2$. We have shown in Section~\ref{S:ub}, and above, that
\begin{equation}\label{E:pi3}
  q^{-2}\left(1-2q^{-1}\right)\leq\pi_3
    \leq q^{-2}\left(1+\frac{58q^{-1}}{9}\right)
    \qquad\textup{for $n\geq1$.}
\end{equation}
Adding $\pi_1=0$ and $\frac{2q^{-3}}{3}\leq\pi_2\leq\frac{8q^{-3}}{3}$ and
$q^{-2}-2q^{-3}\leq\pi_3 \leq q^{-2}+\frac{58q^{-3}}{9}$ gives
$q^{-2}\left(1-\frac{4q^{-1}}{3}\right)\leq\pi\leq q^{-2}\left(1+\frac{82q^{-1}}{9}\right)$
when $r=1$.

Now consider the case when $2\leq r\leq n-r$. Then
$\frac1{12}\leq\frac{2q^{-3}}{3}\leq\pi_1\leq\frac{8q^{-3}}{3}\leq\frac13$
holds by~(\ref{E:NP}). However, $\pi_2$ equals $1-\pi_1$ times the probability
that $B$ is non-cyclic, and hence
\[
  \frac{4q^{-3}}{9}\leq\left(1-\pi_1\right)\frac{2q^{-3}}{3}\leq
  \pi_2\leq\left(1-\pi_1\right)\frac{8q^{-3}}{3}
  \leq\frac{88q^{-3}}{36}.
\]
Adding $\frac{2q^{-3}}{3}\leq\pi_1\leq\frac{8q^{-3}}{3}$ and
$\frac{4q^{-3}}{9}\leq\pi_2\leq\frac{88q^{-3}}{36}$ to the bounds (\ref{E:pi3})
for $\pi_3$ gives
\[
  q^{-2}\left(1-\frac{8q^{-1}}{9}\right)
  \leq\pi_1+\pi_2+\pi_3\leq q^{-2}\left(1+\frac{104q^{-1}}{9}\right)
  \qquad\textup{for $r\geq2$.}
\]
The constants $c_1=-\frac{4}{3}$ and $c_2=\frac{35}{3}$ suffice
as $-\frac{4}{3}<-\frac{8}{9}$ and $\frac{82}{9}<\frac{104}{9}<\frac{35}{3}$.
\end{proof}

\end{document}